\documentclass{article}
\usepackage{geometry}
\usepackage{titling}
\AtEndDocument{\bigskip{\footnotesize%
\textsc{$^*$School of Computing, Dehradun Institute of Technology, Dehradun, Uttarakhand–248009, India} \par  
  \textit{E-mail address:} \texttt{suryagiri456@gmail.com} \par

  \addvspace{\medskipamount}
}}
\usepackage{doc}

\usepackage{url}

\usepackage{graphicx}
\usepackage{epstopdf}
\usepackage{amsthm}
\usepackage{amssymb}
\usepackage{amsmath}
\usepackage{enumitem}

\usepackage{array}
\usepackage{enumitem}
\usepackage[utf8]{inputenc}
\usepackage[english]{babel}
\newtheorem{theorem}{Theorem}

\newtheorem{lemma}[theorem]{Lemma}

\newtheorem*{remark}{Remark}

\DeclareMathOperator{\RE}{Re}

\usepackage{varwidth}
\usepackage{lineno, hyperref}
\usepackage{graphicx}
\usepackage{epstopdf}
\usepackage{amsmath}
\usepackage{amsthm}
\usepackage{enumitem}
\usepackage{amsthm}
\usepackage{float}
\usepackage{subcaption}
\usepackage{caption}
\usepackage{authblk}
\usepackage{abstract}
\usepackage{thmtools}
\declaretheorem[numbered=no,
name=Theorem A]{theoremA}

\begin{document}
\title{Generalized Zalcman Conjecture for Starlike Mappings in Several Complex Variables}
\author{Surya Giri$^*$  }


\date{}


	

\maketitle	

\begin{abstract}
    \noindent    Generalizing the Zalcman conjecture given by $\vert a_n^2 - a_{2n-1}\vert \leq (n-1)^2$,  Ma proposed and proved that the inequality
   $$\vert a_n a_m-a_{n+m-1}\vert \leq (n-1)(m-1), \quad m,n \in \mathbb{N},$$
      holds for functions $f(z)=z+a_2z^2 +a_3 z^3 +\cdots\in \mathcal{S}^*$, the class of starlike functions in the open unit disk.  In this work, we extend this problem to several complex variables for $m=2$ and $n=3$, considering the class of starlike mappings defined on the unit ball in a complex Banach space and on bounded starlike circular domains in $\mathbb{C}^n$.
\end{abstract}
\vspace{0.5cm}

\section{Introduction}\label{sec1}
    Let $\mathcal{S}$ be the class of analytic univalent functions $f$ defined on the open unit disk $\mathbb{U}$ having the series expansion
    $f(z)=z+a_2 z^2 +a_3 z^3+\cdots$, $ z\in \mathbb{U}$.
    By $\mathcal{S}^*$, we denote the subclass of $\mathcal{S}$ consisting of starlike functions with respect to the origin. A function $f\in\mathcal{S}^*$ if and only if $\RE (z f'(z)/f(z))>0$ for $z\in \mathbb{U}$.     Let $\mathcal{P}$ be the class of analytic functions $p$ in $\mathbb{U}$ satisfying $p(0)=1$ and $\RE p(z)>0 $.
    In 1960, as an approach to prove the Bieberbach conjecture, Zalcman conjectured that
    $$ \vert a_n^2 - a_{2n-1}\vert \leq (n-1)^2, \quad n \geq 2 $$
    holds for every $f\in \mathcal{S}$.
    In 2010, Krushkal~\cite{Kru1} established the conjecture for $n\leq 6$  using holomorphic homotopy of univalent functions.
    Employing methods from complex geometry and  the theory of universal Teichm\"{u}ller spaces, Krushkal later claimed to have proved the conjecture for all $n\geq 2$ in his unpublished work~\cite{Kru2}. Meanwhile, the Zalcman conjecture and its generalized version were verified  for various subclasses of $\mathcal{S}$~\cite{BroTsa,LiPon,Ma,RavShe}. In particular, Brown and Tsao~\cite{BroTsa} proved the conjecture in 1986 for starlike and typically real functions and  Ma~\cite{Ma} established the result in 1988 for the class of close-to-convex functions.

    In 1999, Ma~\cite{Ma2} proposed the generalized Zalcman conjecture for $f\in \mathcal{S}$, given by
\begin{equation*}
   \vert J_{m,n}(f)\vert := \vert  a_n a_m - a_{n+m-1}\vert \leq (n-1)(m-1), \quad m,n \geq 2.
\end{equation*}
    Although he proved this conjecture for starlike functions and functions in $\mathcal{S}$ with real coefficients, the problem remains open for the entire class $\mathcal{S}$. Ravichandran and Verma~\cite{RavShe} also established the conjecture for several subclasses of $\mathcal{S}$. The following result is a direct consequence of~\cite[Theorem 2.3]{Ma2}.
\begin{theoremA}\label{thmA}\cite{Ma2}
    If $f(z)=z+\sum_{n=2}^\infty a_n z^n \in \mathcal{S}^*$, then $\vert a_2 a_3-a_4\vert \leq 2$. The estimate is sharp.
\end{theoremA}
      The failure of the Bieberbach conjecture and the Riemann mapping theorem in several complex variables~(see \cite{GraKoh}) suggests that classical results from geometric function theory in one complex variable cannot be directly extended to higher-dimensional cases. Xu~\cite{Xu} studied the Fekete-Szeg\"{o} problem for starlike mappings defined on the unit ball in a complex Banach space and on the unit polydisk in $\mathbb{C}^n$. Giri and Kumar~\cite{GirKum2,GirKum3} derived sharp estimates for second and third-order Toeplitz determinants for the same class defined on the unit ball in a complex Banach space, on the unit polydisk and on bounded starlike circular domains in $\mathbb{C}^n$. For some best-possible results concerning coefficient estimates, Toeplitz determinants and the Fekete--Szeg\"{o} functional in higher dimensions, we refer to~\cite{HamHonKoh,HamKohKoh,Xu2,GirKum1} and the references cited therein.

        In this paper, we study the generalized Zalcman conjecture for starlike mappings in higher dimensions. Extending Theorem A to several complex variables, we establish sharp estimates of $\vert J_{2,3}(f)\vert$ for starlike mappings defined on the unit ball in a complex Banach space and on bounded starlike circular domains in $\mathbb{C}^n$.

        Let $X$ be a complex Banach space equipped with a norm $\|\cdot\|$ and let $\mathbb{B}$ denote the unit ball in $X$. Let $\mathcal{H}(X,Y)$ denote the class of all holomorphic mappings from $X$ into another Banach space $Y$. We simply write $\mathcal{H}(\mathbb{B},X):=\mathcal{H}(\mathbb{B})$. If $f\in \mathcal{H}(\mathbb{B})$, then for each $n\in \mathbb{N}$, there exists a bounded symmetric $n-$linear mapping from $ \prod_{j=1}^n X$ into $X$ such that
        $$f(w)= \sum_{n=0}^\infty \frac{1}{n!} D^n f(z)((w-z)^n)$$
        for all $w$ in some neighbourhood of $z\in \mathbb{B}$, where $D^n f(z)$ denotes the $n-$th Fr\'{e}chet derivative of $f$ at $z$. Moreover,
        $ D^n f(z)((w-z)^n)= D^n f(z) \underbrace{( w-z, w-z, \cdots, w-z) }_\text{ n -times}.$
        For each $z\in X\setminus\{0\}$, define
         $$ T_z = \{ l_z \in L(X,\mathbb{C}) : l_z(z) = \| z\|, \| l_z \| = 1\},$$
       where $L(X,Y)$ denotes the space of continuous linear operators from $X$ into $Y$. This set is non-empty by the Hahn-Banach theorem. A mapping $f\in \mathcal{H}(\mathbb{B})$ is said to be biholomorphic if its inverse $f^{-1}$ exists and is holomorphic on $f(\mathbb{B})$. If the Fr\'{e}chet derivative $Df(z)$ of $f\in \mathcal{H}(\mathbb{B})$ has a bounded inverse for each $z\in \mathbb{B}$, then $f$ is called locally biholomorphic. A mapping $f\in \mathcal{H}(\mathbb{B})$ is said to be normalized if $f(0)=0$ and $Df(0)=I$, where $I \in L(X,X)$.

       Let $\Omega \subset \mathbb{C}^n$ be a bounded starlike circular domain with $0\in \Omega$ and let its Miknowski functional $\rho(z)\in \mathcal{C}^1$ (see Lemma~\ref{Minkow}), except on some lower dimensional manifolds in $\mathbb{C}^n$. Let $\partial \Omega$ denote the boundary of $\Omega$. The first and the $m^{th}$ Fr\'{e}chet derivative of a holomorphic mapping $f : \Omega \rightarrow X$  are written by $ D f(z)$ and $D^m f(z) (a^{m-1},\cdot)$, respectively. The matrix representations are
\begin{align*}
    D f(z) &= \bigg(\frac{\partial f_j}{\partial z_k} \bigg)_{1 \leq j, k \leq n}, \\
    D^m f(z)(a^{m-1}, \cdot) &= \bigg( \sum_{p_1,p_2, \cdots, p_{m-1}=1}^n  \frac{ \partial^m f_j (z)}{\partial z_k \partial z_{p_1} \cdots \partial z_{p_{m-1}}} a_{p_1} \cdots a_{p_{m-1}}   \bigg)_{1 \leq j,k \leq n},
\end{align*}
   where $f(z) = (f_1(z), f_2(z), \cdots f_n(z))'$ and $ a= (a_1, a_2, \cdots a_n)'\in \mathbb{C}^n.$
   We need the following lemmas to prove our main results.
\begin{lemma}\cite{Suf}
  Let $f: \mathbb{B} \rightarrow X$ be a normalized locally biholomorphic mapping. The mapping $f$ is said to be starlike on $\mathbb{B}$ if and only if
    $$ \RE (l_z [Df(z)]^{-1}f(z))> 0, \quad x\in \mathbb{B}\setminus \{0\}, \quad l_z \in T_z. $$
    The class of all such mappings on $\mathbb{B}$ is denoted by $\mathcal{S}^*(\mathbb{B})$.
\end{lemma}
\begin{lemma}\label{Minkow}\cite{LiuRen}
  $\Omega \subset \mathbb{C}^n$  is said to be a bounded starlike circular domain if and only if there exists a unique continuous function
 $\rho: \mathbb{C}^n \rightarrow \mathbb{R},$ called the Minkowski functional of $\Omega$, such that
\begin{enumerate}
    \item[(i)] $\rho(z) \geq 0$, $z \in \mathbb{C}^n$; \quad $\rho(z) = 0  \Leftrightarrow z=0$,
    \item[(ii)] $\rho(t z) = \vert t\vert , \rho(z)$, $t \in \mathbb{C}$, $z \in \mathbb{C}^n$,
    \item[(iii)] $\Omega = \{ z \in \mathbb{C}^n : \rho(z) < 1 \}$.
\end{enumerate}
   Furthermore, if $\rho(z)$, $z \in \Omega$, belongs to $\mathcal{C}^1$ except on some lower-dimensional manifold $E \subset \mathbb{C}^n$, then $\rho(z)$ satisfies the following properties:
\begin{align}\label{LmEqn}
   & 2 \frac{\partial \rho(z)}{\partial z} z = \rho(z), \quad z \in \mathbb{C}^n \setminus E, \\
    & 2 \frac{\partial \rho(z)}{\partial z} \bigg\vert_{z=z_0} = 1, \quad z_0 \in \partial \Omega \setminus E, \notag\\
    & \frac{\partial \rho(\lambda z)}{\partial z} = \frac{\partial \rho(z)}{\partial z}, \quad \lambda \in (0,\infty),\; z \in \mathbb{C}^n \setminus E, \notag\\
    & \frac{\partial \rho(e^{i \theta} z)}{\partial z} = e^{-i\theta} \frac{\partial \rho(z)}{\partial z}, \quad \theta \in \mathbb{R},\ z \in \mathbb{C}^n \setminus E, \notag
\end{align}
where
   $\frac{\partial \rho(z)}{\partial z} = \left( \frac{\partial \rho(z)}{\partial z_1}, \dots, \frac{\partial \rho(z)}{\partial z_n} \right).$
\end{lemma}

\begin{lemma}\cite{LiuRen}
   Let $\Omega \subset \mathbb{C}^n$ be a bounded starlike circular domain with $0\in \Omega$, whose Minkowski functional $\rho(z)$ belongs to $\mathcal{C}^1$ except on some lower-dimensional manifolds $E \subset \mathbb{C}^n$. Let $f:\Omega \to \mathbb{C}^n$ be a normalized locally biholomorphic mapping. Then $f$ is starlike on $\Omega$ if and only if
  $$\RE \bigg( \frac{\partial \rho(z)}{\partial z}\,(Df(z))^{-1} f(z) \bigg) > 0,\quad z \in \Omega \setminus E. $$
  The class of all starlike mappings on $\Omega$ is denoted by $\mathcal{S}^*(\Omega)$.
\end{lemma}
\begin{lemma}\cite{ChoKumRav}\label{Carath}
    Let $p(z)=1+\sum_{n=1}^\infty p_n z^n \in \mathcal{P}$. Then the following holds:
    $$\vert p_n \vert\leq 2, \;\; \vert p_2 - p_1^2 \vert \leq 2\;\; \text{and}\;\; \vert p_3 - p_1 p_2 \vert\leq 2. $$
\end{lemma}

\section{Main results}
   This section is devoted to obtaining the sharp bound of $\vert J_{2,3}(f)\vert$ for starlike mappings in higher dimensions.
\begin{theorem}\label{thmB}
    Let $f\in \mathcal{H}(\mathbb{B},\mathbb{C})$ with $f(0)=1$ and suppose that $F(z)= z f(z)$. If $ F(z) \in \mathcal{S}^*(\mathbb{B})$,   then
\begin{equation*}
\begin{aligned}
     \bigg\vert \bigg( \frac{ l_z (D^2 F(0) (z^2))}{2! \vert\vert z \vert\vert^2} \bigg) \bigg( \frac{ l_z (D^3 F(0) (z^3))}{3! \vert\vert z \vert\vert^3} \bigg) - \bigg(\frac{ l_z (D^4 F(0) (z^4))}{4! \vert\vert z \vert\vert^4}& \bigg) \bigg\vert  \leq 2.
\end{aligned}
\end{equation*}
   The bound is sharp.
\end{theorem}
\begin{proof}
     Let the function $ h : \mathbb{U} \rightarrow \mathbb{C}$ be defined by
\begin{equation*}
    h(\zeta) = \left\{ \begin{array}{ll}
     \dfrac{\zeta}{ l_z ((D F(\zeta z_0))^{-1} F( \zeta z_0) )}, & \zeta \neq 0, \\ \\
    1, & \zeta =0
    \end{array}
    \right.
\end{equation*}
    for fixed $z\in X\setminus \{ 0 \}$ and $z_0 = \frac{z}{\|z \|}$.
   Then $h \in \mathcal{H}(\mathbb{U})$. Since $\RE ((D F(z))^{-1} F(z))>0$, we have
\begin{equation}
   \RE h(\zeta)> 0.
\end{equation}
       Following the same technique used in \cite[Theorem 7.1.14]{GraKoh}, we obtain
    $$ (D F(z))^{-1} = \frac{1}{f(z)} \bigg( I - \frac{\frac{z D f(z)}{f(z)}}{1 + \frac{D f(z) z}{f(z)}} \bigg),  \quad z\in \mathbb{B}. $$
    Consequently, we get
    $$ (D F(z))^{-1} F(z) = z \bigg( \frac{z f(z) }{f(z) + D f(z) z} \bigg), $$
    which further implies
\begin{equation*}\label{newe}
   \frac{\| z\|}{l_z ((D F(z))^{-1} F(z))} = 1 + \frac{D f(z) z}{f(z)} .
\end{equation*}
  From the above equation, we deduce that
\begin{equation*}\label{accr}
    h(\zeta) = \frac{\| \zeta z_0 \| }{l_{ \zeta z_0} ((Df(\zeta z_0))^{-1} f( \zeta z_0) ) }  =  1 + \frac{D f(\zeta z_0)\zeta z_0}{f(\zeta z_0)}.
\end{equation*}
   By expanding the Taylor series of $h(\zeta)$ and $f(\zeta z_0)$, the above equation leads to
\begin{align*}
   \bigg(1 + & h'(0) \zeta  + \frac{h''(0)}{2} \zeta^2 + \cdots \bigg)\bigg( 1 + Df(0)(z_0) \zeta + \frac{ D^2 f(0)(z_{0}^2)}{2} \zeta^2 + \cdots \bigg)  \\
   & =\bigg( 1 + Df(0)(z_0) \zeta + \frac{ D^2 f(0)(z_{0}^2)}{2} \zeta^2 + \cdots \bigg) \bigg( Df(0)(z_0) \zeta +  D^2 f(0)(z_{0}^2)\zeta^2 + \cdots \bigg).
\end{align*}
   By comparing the homogeneous expansions, we obtain
\begin{align*}
    h'(0) &= D f(0)(z_0),\;\; \frac{h''(0)}{2} =  D^2 f(0)(z_0^2) - (D f(0)(z_0))^2
\end{align*}
   and
   $$     \frac{h'''(0)}{6} = ( Df(0)(z_0))^3 - \frac{3}{2} D f(0)(z_0) D^2 f(0) (z_0^2) + \frac{D^3 f(0) (z_0^3)}{2}.$$
   That is
\begin{equation}\label{eqhf2}
\begin{aligned}
\left.
\begin{array}{ll}
     h'(0) \|z\|&= D f(0)(z), \;\;    \dfrac{h''(0)}{2}\| z\|^2 =  D^2 f(0)(z^2) - (D f(0)(z))^2, \\ \\
    \dfrac{h'''(0)}{6} \|z\|^3 &=  ( Df(0)(z))^3 -\dfrac{3}{2} D f(0)(z) D^2 f(0) (z^2) + \dfrac{D^3 f(0) (z^3)}{2}.
\end{array}
\right\}
\end{aligned}
\end{equation}
     Moreover, from the relation $F(z) = z f(z)$, it follows that
\begin{align*}
     \frac{ D^2 F(0) (z^2)}{2! } &=  D f(0)(z)  z,\;\;  \frac{ D^3 F(0) (z^3)}{3! } =  \frac{ D^2 f(0) (z^2)}{2! } z\;\; \text{and}\;\;  \frac{ D^4 F(0) (z^4)}{4! } =  \frac{ D^3 f(0) (z^3)}{3! } z,
\end{align*}
   which further leads to
\begin{align*}
     \frac{l_z( D^2 F(0) (z^2))}{2! } &=  D f(0)(z)  \|z\|, \;\;\frac{l_z (D^3 F(0) (z^3))}{3! } =  \frac{ D^2 f(0) (z^2)}{2! } \|z\|
\end{align*}
 and
   $$          \frac{l_z( D^4 F(0) (z^4))}{4! } =  \frac{ D^3 f(0) (z^3)}{3! } \|z\|,$$
   respectively.   In view of (\ref{eqhf2}), the above expressions reduce to
\begin{equation}\label{T1E3}
     \frac{l_z( D^2 F(0) (z^2))}{2! \|z\|^2} = h'(0), \;\;    \frac{l_z (D^3 F(0) (z^3))}{3! \|z\|^3} = \frac{1}{2}\bigg(\frac{h''(0)}{2}+ (h'(0))^2 \bigg)
\end{equation}
   and
\begin{equation}\label{T1E4}
    \frac{l_z( D^4 F(0) (z^4))}{4! \|z\|^3} =  \frac{1}{3} \bigg( \frac{h'''(0)}{6}+\frac{(h'(0))^3}{2}+ \frac{3 h'(0) h''(0)}{4} \bigg).
\end{equation} 
   From (\ref{T1E3}) and (\ref{T1E4}), we deduce that
\begin{align*}
     \bigg\vert \bigg( \frac{ l_z (D^2 F(0) (z^2))}{2! \vert\vert z \vert\vert^2} \bigg)& \bigg( \frac{ l_z (D^3 F(0) (z^3))}{3! \vert\vert z \vert\vert^3} \bigg) - \bigg(\frac{ l_z (D^4 F(0) (z^4))}{4! \vert\vert z \vert\vert^4}\bigg) \bigg\vert \\
     &\quad \quad\quad\quad= \frac{1}{3} \bigg\vert (h'(0))^3 - \frac{h'''(0)}{6} \bigg\vert \\
     &\quad \quad\quad\quad=\frac{1}{3} \bigg\vert h'(0) \Big( (h'(0))^2- \frac{h''(0)}{2} \Big) +  \frac{h'(0) h''(0)}{2} -\frac{h'''(0)}{6} \bigg\vert\\
     &\quad \quad\quad\quad\leq \frac{1}{3} \bigg( \vert h'(0)\vert \Big\vert (h'(0))^2- \frac{h''(0)}{2} \Big\vert + \Big\vert \frac{h'(0) h''(0)}{2} -\frac{h'''(0)}{6} \Big\vert \bigg).
\end{align*}
     Since $h\in\mathcal{P}$, applying Lemma~\ref{Carath} to the above inequality yields the asserted bound.
   
   To complete the sharpness part, let us consider the mapping $F$ given by
\begin{equation}\label{extB}
    F(z) = \frac{z}{(1-(l_{u}(z))^2}, \quad z\in \mathbb{B}, \quad \vert\vert u \vert\vert=1.
\end{equation}
   It is evident that $ F \in \mathcal{S}^*(\mathbb{B})$. A straightforward calculation yields
  $$  \frac{D^2  F(0) (z^2)}{2!}=  2 l_u(z) z , \;\; \frac{D^3  F(0) (z^3)}{3!} = 3 (l_u (z))^2 z \;\; \text{and}\;\;   \frac{D^4  F(0) (z^4)}{4!}  =4 (l_u(z))^3 z, $$
   which immediately provide
  $$  \frac{l_z(D^2  F(0) (z^2))}{2!}=  2 l_u(z) \|z\|, \;\;\; \frac{l_z (D^3  F(0) ( z^3 ))}{3!} = 3 (l_u (z))^2 \| z \|  $$
   and
\begin{align*}
     \frac{l_z (D^4  F(0) (z^4))}{4!} & =4 (l_u(z))^3 \|z\|,
\end{align*}
   respectively. Setting $z = r u$ $(0< r <1)$, it follows that
\begin{equation*}\label{cftB}
     \frac{l_z(D^2  F(0) (z^2))}{2! \|z\|^2}=  2, \;\;\; \frac{l_z (D^3  F(0) ( z^3) ) }{3! \| z \|^3}  =  3 \;\; \text{and}\;\;   \frac{l_z (D^4  F(0) (z^4))}{4! \| z \|^4}  =4.
\end{equation*}
   Consequently, for the mapping $F$, we have
\begin{align*}
    \bigg\vert \bigg( \frac{ l_z (D^2 F(0) (z^2))}{2! \vert\vert z \vert\vert^2} \bigg) \bigg( \frac{ l_z (D^3 F(0) (z^3))}{3! \vert\vert z \vert\vert^3} \bigg) - \bigg(\frac{ l_z (D^4 F(0) (z^4))}{4! \vert\vert z \vert\vert^4}& \bigg) \bigg\vert  = 2,
\end{align*}
  establishing the sharpness of the bound.
\end{proof}
\begin{remark}
   For $X=\mathbb{C}$ and $\mathbb{B}=\mathbb{U}$, Theorem~\ref{thmB} is equivalent to Theorem A.
\end{remark}
\begin{theorem}\label{ThmUn1}
     Let $f \in \mathcal{H}(\Omega, \mathbb{C})$ with $f(0)=1$  and suppose that $F(z) =  z f(z)$. If $F(z) \in \mathcal{S}^*(\Omega)$,
   then
\begin{equation*}\label{mnresult}
\begin{aligned}
   \bigg\vert \bigg(2 \frac{\partial \rho(z)}{\partial z}\frac{D^{2} F(0)(z^{2})}{2! \rho^{2}(z)}\bigg)\bigg(2 \frac{\partial \rho(z)}{\partial z}\frac{D^{3} F(0)(z^{3})}{3! \rho^{3}(z)}\bigg) -2 \frac{\partial \rho(z)}{\partial z}\frac{D^{4} F(0)(z^{4})}{4! \rho^{4}(z)}\bigg\vert \leq 2 , \quad z \in \Omega\setminus E .
\end{aligned}
\end{equation*}
   The bound is sharp.
\end{theorem}
\begin{proof}
  For $z \in \Omega \setminus E $, set $z_0 = \frac{z}{\rho(z)}$ . Define $s : \mathbb{U} \rightarrow \mathbb{C}$ by
\begin{equation}\label{hkzeta}
   s (\zeta) =
\left\{
\begin{array}{ll}
    \dfrac{\zeta }{2 \frac{\partial \rho(z_0)}{\partial z} (D F(\zeta z_0))^{-1} F(\zeta z_0)}, & \zeta \neq 0,\\
     1 , & \zeta =0.
\end{array}
\right.
\end{equation}
   Then $s\in \mathcal{H}(\mathbb{U})$. Furthermore, since $F \in \mathcal{S}^*(\Omega)$, it follows that
\begin{align*}
    \RE s (\zeta) &= \RE \bigg( \frac{\zeta }{2 \frac{\partial \rho(z_0)}{\partial z} (D F(\zeta z_0))^{-1} F(\zeta z_0)}\bigg)  \\
                  &= \RE \bigg( \frac{\rho (\zeta z_0) }{2 \frac{\partial \rho(\zeta z_0)}{\partial z} (D F(\zeta z_0))^{-1} F(\zeta z_0)}\bigg)> 0,\quad  \zeta \in \mathbb{U}.
\end{align*}
   Following the same procedure as in Theorem~\ref{thmB}, we deduce that
   $$ (D F(z))^{-1} F(z) = z \bigg( \frac{z f(z) }{f(z) + D f(z) z} \bigg),  \quad z \in \Omega \setminus\{0\},$$
   which in turn yields
\begin{equation}\label{thm2eq1}
    \frac{\rho ( z) }{2 \frac{\partial \rho(z)}{\partial z} (D F(z))^{-1} F(z)} = 1 + \frac{D f(z) z}{f(z)} , \quad z \in \Omega\setminus\{E\}.
\end{equation}
    By virtue of (\ref{thm2eq1}), we get
    $$ s(\zeta)=  1 + \frac{D f(\zeta z_0) \zeta z_0}{f(\zeta z_0)}. $$
   Expanding $f$ and $s$ in Taylor series and equating like homogeneous terms, we obtain
\begin{equation*}
    s' (0) = D f(0)(z_0), \quad \frac{s'' (0)}{2} = D^2 f(0) (z_0^2 ) - (D f(0) (z_0))^2
\end{equation*}
  and
\begin{equation*}
    \frac{s'''(0)}{6} = ( Df(0)(z_0))^3 - \frac{3}{2}D f(0)(z_0) D^2 f(0) (z_0^2)  + \frac{D^3 f(0) (z_0^3)}{2},
\end{equation*}
   which further gives
\begin{equation}\label{use0}
    s' (0) \rho(z)= D f(0)(z), \quad \frac{s'' (0)}{2} \rho^2(z) = D^2 f(0) (z^2 ) - (D f(0) (z))^2
\end{equation}
  and
\begin{equation}\label{hing4}
    \frac{s'''(0)}{6} \rho^3(z)= ( Df(0)(z))^3 - \frac{3}{2}D f(0)(z) D^2 f(0) (z^2)  + \frac{D^3 f(0) (z^3)}{2},
\end{equation}
  respectively. Moreover, from the relation $F(z) =  z f(z)$, we have
\begin{equation}\label{use1}
      \frac{D^2 F(0) (z^3)}{2!}= Df(0)(z) z, \quad    \frac{D^3 F(0) (z^3)}{3!} = \frac{D^2 f(0) (z^2)}{2!} z
\end{equation}
    and
\begin{equation}\label{use2}
     \frac{D^4 F(0) (z^4)}{4!} =\frac{ D^3 f(0) (z^3)}{3! }  z.
\end{equation}
   Using~(\ref{LmEqn}) in (\ref{use1}) and (\ref{use2}), it follows that
\begin{equation}\label{use3}
     2 \frac{\partial \rho}{\partial z} \frac{D^2 F(0) (z^3)}{2!}= Df(0)(z) \rho(z), \quad     2 \frac{\partial \rho}{\partial z}\frac{D^3 F(0) (z^3)}{3!} = \frac{D^2 f(0) (z^2)}{2!} \rho(z)
\end{equation}
    and
\begin{equation}\label{use4}
      2 \frac{\partial \rho}{\partial z} \frac{D^4 F(0) (z^4)}{4!} =\frac{ D^3 f(0) (z^3)}{3! }  \rho(z),
\end{equation}
   respectively. In view of (\ref{use0}), (\ref{hing4}), (\ref{use3}) and (\ref{use4}), we deduce that
\begin{align*}
   \bigg\vert \bigg(2 \frac{\partial \rho(z)}{\partial z}\frac{D^{2} F(0)(z^{2})}{2! \rho^{2}(z)}\bigg)&\bigg(2 \frac{\partial \rho(z)}{\partial z}\frac{D^{3} F(0)(z^{3})}{3! \rho^{3}(z)}\bigg) -2 \frac{\partial \rho(z)}{\partial z}\frac{D^{4} F(0)(z^{4})}{4! \rho^{4}(z)}\bigg\vert \\
         &\quad \quad\quad \quad= \frac{1}{3} \bigg\vert (s'(0))^3 - \frac{s'''(0)}{6} \bigg\vert\\
    &\quad \quad\quad\quad=\frac{1}{3} \bigg\vert s'(0) \Big( (s'(0))^2- \frac{s''(0)}{2} \Big) +  \frac{s'(0) s''(0)}{2} -\frac{s'''(0)}{6} \bigg\vert\\
     &\quad \quad\quad\quad\leq \frac{1}{3} \bigg( \vert s'(0)\vert \Big\vert (s'(0))^2- \frac{s''(0)}{2} \Big\vert + \Big\vert \frac{s'(0) s''(0)}{2} -\frac{s'''(0)}{6} \Big\vert \bigg).
\end{align*}
   Using Lemma~\ref{Carath} in the above inequality, we get the required bound.
   
      To verify the sharpness of the bound, consider the mapping
\begin{equation}\label{ExtOmega}
   F(z)= \frac{z}{\Big(1 -  \left(\dfrac{z_1}{r} \right)\Big)^{2}}, \quad z\in \Omega,
\end{equation}
   where $r=\sup\{\vert z_1\vert : z=(z_1,0,\cdots,0)'\in \Omega \}$. It follows from~\cite{LiuLiu2} that the mapping given by~(\ref{ExtOmega}) belongs to $\mathcal{S}^*(\Omega)$. For this mapping, a direct computation yields
    $$  \frac{D^{2}  F(0) (z^{2})}{2!}= 2  \Big(\frac{z_1}{r}\Big) z, \;\;  \frac{D^{3}  F(0) (z^{3})}{3!} = 3 \Big(\frac{z_1}{r}\Big)^{3} z \;\; \text{and} \;\; \frac{D^{4}  F(0) (z^{4})}{4!} = 4 \Big(\frac{z_1}{r}\Big)^{4} z.$$
    Using (\ref{LmEqn}) in the above equations, we obtain
    $$ 2 \frac{\partial \rho}{\partial z}  \frac{D^{2}  F(0) (z^{2})}{2!}= 2  \Big(\frac{z_1}{r}\Big) \rho (z), \;\;  2 \frac{\partial \rho}{\partial z} \frac{D^{3}  F(0) (z^{3})}{3!} = 3 \Big(\frac{z_1}{r}\Big)^{3}\rho (z)  $$
    and
   $$  2 \frac{\partial \rho}{\partial z} \frac{D^{4}  F(0) (z^{4})}{4!} = 4 \Big(\frac{z_1}{r}\Big)^{4} \rho (z), $$
   respectively.  Taking $z = R u$ $(0< R <1)$, where $u =(u_1, u_2, \cdots, u_n)' \in \partial \Omega$ and $u_1 =r$, we deduce that
  $$ 2 \frac{\partial \rho}{\partial z}  \frac{D^{2}  F(0) (z^{2})}{2! \rho^2(z)}= 2  , \;\;  2 \frac{\partial \rho}{\partial z} \frac{D^{3}  F(0) (z^{3})}{3! \ \rho^3(z)} = 3 \;\; \text{and} \;\; \frac{D^{4}  F(0) (z^{4})}{4! \rho^4(z)} = 4. $$
   Hence, it follows that
\begin{align*}
        \bigg\vert \bigg(2 \frac{\partial \rho(z)}{\partial z}\frac{D^{2} F(0)(z^{2})}{2! \rho^{2}(z)}\bigg)\bigg(2 \frac{\partial \rho(z)}{\partial z}\frac{D^{3} F(0)(z^{3})}{3! \rho^{3}(z)}\bigg) -2 \frac{\partial \rho(z)}{\partial z}\frac{D^{4} F(0)(z^{4})}{4! \rho^{4}(z)}\bigg\vert = 2 ,
\end{align*}
    which confirms the sharpness of the bound.
\end{proof}
\begin{remark}
    For $n=1$ and $\Omega=\mathbb{U}$, Theorem~\ref{ThmUn1} reduces to Theorem A.
\end{remark}

\section*{Declarations}
\subsection*{Conflict of interest}
	The authors declare that they have no conflict of interest.
\subsection*{Data Availability} Not Applicable.

\end{document}